\pdfoutput=1
\documentclass[12pt]{amsart}
\usepackage{fullpage}
\usepackage{microtype}
\usepackage{stmaryrd}

\usepackage[utf8]{inputenc}
\usepackage{amsfonts}
\usepackage{amsmath}
\usepackage{mathtools}
\usepackage{fancyhdr}
\usepackage{graphicx}
\usepackage{enumitem}
\usepackage{stackrel,amssymb}
\usepackage{stmaryrd}
\usepackage{bm}
\usepackage{tikz}
\usetikzlibrary{calc}
\usetikzlibrary{cd}
\usepackage{setspace}
\usepackage{multicol}
\usepackage{romannum}
\usepackage{amsthm} %
\usepackage[bookmarksdepth=3]{hyperref}
\usepackage{cleveref}
\usepackage{cite}
\usepackage{mathtools}
\usepackage{mathrsfs}
\usepackage{booktabs,rotating}
\input{format_basic.tex}

\newcommand{\Z}{\mathbb{Z}}
\newcommand{\Q}{\mathbb{Q}}

\newcommand{\R}{\mathbb{R}}

\DeclareMathOperator{\coker}{coker}
\newcommand{\Gal}{\text{Gal}}
\newcommand{\Hom}{\text{Hom}}

\newcommand{\Ext}{\text{Ext}}
\newcommand{\Tor}{\text{Tor}}

\DeclareMathOperator{\tr}{tr}
\newcommand{\id}{\text{id}}

\renewcommand{\to}{\longrightarrow}

\renewcommand{\mapsto}{\longmapsto}
\newcommand{\onto}{\twoheadrightarrow}
\newcommand{\into}{\xhookrightarrow{}}

\newcommand{\simto}{\stackrel{\sim}{\to}}

\renewcommand{\P}{\mathfrak{P}}

\newcommand{\m}{\mathfrak{m}}

\newcommand{\GL}{\text{GL}}
\newcommand{\SL}{\text{SL}}

\newcommand{\blank}{{-}}

\renewcommand{\bar}{\overline}
\renewcommand{\tilde}{\widetilde}

\usepackage{todonotes}

\newcommand*\fixitem {\item[]%
  \refstepcounter{enumi}\hskip-\leftmargin\labelenumi\hskip\labelsep}

\crefname{equation}{}{}
\Crefname{equation}{}{}
\crefname{enumi}{}{}
\Crefname{enumi}{}{}

\newcommand{\questionbox}{{$\square$\kern-0.58em{\raisebox{0.08em}{\tiny\textbf{?}}}}}

\newcommand{\A}{\mathbb{A}}
\renewcommand{\P}{\mathbb{P}}
\newcommand{\F}{\mathbb{F}}

\DeclareMathOperator{\Ind}{Ind}

\DeclareMathOperator{\cInd}{c-Ind}

\DeclareMathOperator{\Sym}{Sym}
\DeclareMathOperator{\Map}{Map}
\newcommand{\otimesL}{\otimes^\mathbb{L}}

\newcommand{\ad}{\text{ad}}

\newcommand{\Gcal}{\mathcal{G}}

\newcommand{\Hcal}{\mathcal{H}}

\newcommand{\Tbb}{\mathbb{T}}

\DeclareMathOperator{\St}{St}

\newcommand{\Frob}{\text{Frob}}

\newcommand{\rig}{\text{rig}}

\newcommand{\Brig}[2][]{%
\ifthenelse{\equal{#2}{}}{\mathbf{B}_{\rig,#1}^{\dagger}}{\mathbf{B}_{\rig,#1}^{\dagger, #2}}%
}

\newcommand{\diamondp}{S}
\DeclareMathOperator{\MF}{MF}
\newcommand{\loccit}{\textit{loc.~ cit.}}

\newcommand{\unr}[1]{\mathrm{unr}_{{#1}}}

\usepackage{scalerel}

\begin{document}
\pagenumbering{arabic}

\title{The $p$-arithmetic homology of mod $p$ representations of $\GL_2(\Q_p)$}
\author{Guillem Tarrach}

\begin{abstract}
    We compute the non-Eisenstein systems of Hecke eigenvalues contributing to the $p$-arithmetic homology of irreducible smooth mod $p$ representations $\pi$ of $\GL_2(\Q_p)$ and to the cohomology of their duals.
    We show that in most cases they are associated to odd irreducible 2-dimensional Galois representations whose local component at $p$ corresponds under the mod $p$ local Langlands correspondence to a smooth representation that contains $\pi$ as a subrepresentation.
\end{abstract}

\maketitle

\setcounter{tocdepth}{1}
\tableofcontents

\section{Introduction}

Let $p \geq 5$ be a prime number and $N \geq 3$ an integer coprime to $p$. Let $\Gamma_1^p(N)$ be the subgroup of matrices in $\GL_2(\Z[1/p])$ that have positive determinant and are congruent modulo $N \Z[1/p]$ to a matrix of the form $\begin{pmatrix} * & * \\ 0 & 1\end{pmatrix}$.
The goal of this article is to compute the systems of Hecke eigenvalues contributing to the homology of $\Gamma_1^p(N)$ with coefficients in the irreducible mod $p$ representations of $\GL_2(\Q_p)$ and the cohomology of their duals.
More specifically, we prove the following local-global compatibility result.

\begin{theorem}\label{thm: 1}
    Let $\pi$ be an irreducible smooth mod $p$ representation of $\GL_2(\Q_p)$ over $\bar \F_p$, $\pi^\vee$ its abstract $\bar \F_p$-dual. Then:
    \begin{enumerate}
        \item\label{item: thm 1 i} The $p$-arithmetic homology $H_*(\Gamma_1^p(N), \pi)$ and cohomology $H^*(\Gamma_1^p(N), \pi^\vee)$ are finite-dimensional and vanish in degrees outside the range $[0,3]$.
        \item\label{item: thm 1 ii} To any system of Hecke eigenvalues in these spaces, one can associate a 2-dimensional odd semisimple mod $p$ Galois representation of $\Gal(\bar \Q / \Q)$ satisfying the usual relations at primes not dividing $pN$.

        \item\label{item: thm 1 iii} Let $\rho \colon \Gal(\bar \Q / \Q) \to \GL_2(\bar \F_p)$ be a 2-dimensional odd irreducible Galois representation.
        Then, $\rho$ contributes to $H_*(\Gamma_1^p(N), \pi)$ and $H^*(\Gamma_1^p(N), \pi^\vee)$ if and only if $N$ is a multiple of the minimal level $N(\rho)$ attached to $\rho$ by Serre \cite{serre_modularity}, and one of the following is satisfied:
        \begin{enumerate}[label=(\alph*)]
            \item\label{item: thm 1 iii a} $\pi$ is a subrepresentation of the representation associated to the restriction of $\rho$ at a decomposition group $\Gcal_p$ at $p$ by the mod $p$ local Langlands correspondence for $\GL_2(\Q_p)$.
            In this case, $\rho$ contributes to (co)homology in degrees 1, 2 and 3, unless $\pi$ is a twist of the Steinberg representation, in which case $\rho$ contributes to cohomology in degrees 1 and 2.
            \item\label{item: thm 1 iii b} $\pi$ is a character, say $\pi = \chi \circ \det$, and $\rho|_{\Gcal_p}$ is an extension of $\chi \omega^{-1}$ by $\chi$, where we have identified $\chi$ with a character of $\Gcal_p$ via local class field theory and $\omega$ denotes the mod $p$ cyclotomic character.
            In this case, $\rho$ contributes to (co)homology in degrees 2 and 3.
        \end{enumerate}
    \end{enumerate}
\end{theorem}

The proof of the theorem is obtained by combining the explicit construction of the irreducible mod $p$ representations of $\GL_2(\Q_p)$ due to Barthel--Livn\'e \cite{barthel_livne} and Breuil \cite{breuil_sur_quelques}, a result relating $p$-arithmetic homology to arithmetic homology in the spirit of \cite{kohlhaase_schraen} and \cite{S_arithmetic_cohomology}, and classical results on the weight part of Serre's conjecture.
These are already enough to prove the generic case where $\pi$ is supersingular or principal series.
The cases of (twists of) the trivial and Steinberg representations require more work, and involve the group cohomological analogue of multiplication of mod $p$ modular forms by the Hasse invariant studied by Edixhoven--Khare \cite{edixhoven_khare_hasse} and an interpretation of this map in terms of the representation theory of the local group $\GL_2(\Q_p)$.

\subsection{Notation and conventions}

Write $G = \GL_2(\Q_p)$, $K = \GL_2(\Z_p)$ and $Z$ for the center of $G$, so that $Z \simeq \Q_p^\times$.
Let $\alpha = \begin{pmatrix} 1 & 0 \\ 0 & p \end{pmatrix} \in G$ and $\beta = \begin{pmatrix} p & 0 \\ 0 & p \end{pmatrix} \in Z$.
Write also $B$ for the subgroup of upper-triangular matrices in $G$ and $I = K \cap \alpha K \alpha^{-1}$ for the Iwahori subgroup of matrices in $K$ that are upper-triangular modulo $p$.
Let $G^+ \subseteq G$ be the submonoid of matrices whose entries lie in $\Z_p$, and $G^- = (G^+)^{-1}$.
We will write $\omega$ for the character $\Q_p^\times \to \F_p^\times$ defined by $x \mapsto x |x| \mod p$.
Write $k = \bar \F_p^\times$.

We normalise local class field theory so that uniformisers correspond to geometric Frobenii, and for each prime $\ell$ we let $\Frob_\ell$ be the geometric Frobenius corresponding to $\ell$.
Choose embeddings $\bar \Q \into \bar \Q_\ell$ for all $\ell$, and let $\Gcal_\ell$ denote the corresponding decomposition group at $\ell$ in $\Gal(\bar \Q / \Q)$.
We will use our normalisation of local class field theory to identify characters of $\Gcal_\ell$ and $\Q_\ell^\times$ without coming.
Write $\varepsilon \colon \Gal(\bar \Q / \Q) \to k^\times$ for the mod $p$ cyclotomic character, it satisfies $\varepsilon(\Frob_\ell) = \ell^{-1} \mod p$ and $\varepsilon(\Frob_p) = 1$.
Its restriction to $\Gcal_p$ at $p$ corresponds to $\omega$ under local class field theory.
Write $\mathcal{I}_p \subseteq \Gcal_p$ for the inertia subgroup.
Let $\omega_2 \colon \mathcal{I}_p \to \mu_{p^2 - 1} (\bar \Q_p^\times) \subseteq \bar \F_p^\times$ be Serre's fundamental character of level 2, defined by $\omega_2(g) = (g p^{1 / (p^2-1)}) / p^{1/(p^2-1)}$, and for $0 \leq s \leq p$ let $\Ind(\omega_2^s)$ be the irreducible representation of $\Gcal_p$ over $\chi$ with determinant $\omega^s$ and $\Ind(\omega_2^s)|_{\mathcal{I}_p} = \omega_2^s \oplus \omega_2^{ps}$.
All irreducible 2-dimensional representations of $\Gcal_p$ over $k$ are of the form $\Ind(\omega_2^s) \otimes \chi$ for some $s$ as above and character $\chi$.
Given a two-dimensional odd and irreducible representation $\rho$ of $\Gal(\bar \Q / \Q)$ over $k$, we will write $N(\rho)$ for the minimal level attached to $\rho$ by Serre in \cite{serre_modularity}.

Given $b \in k$, we will write $\unr{b}$ for the $k$-valued unramified characters of $\Q_p^\times$ and of $\Gal(\bar \Q_p / \Q_p)$ sending $p$ and $\Frob_p$ respectively to $b$.
Thus, all continuous characters $\Q_p^\times \to k^\times$ are of the form $\omega^a \unr{b}$ for some $b \in k$ and $0 \leq a \leq p-2$.
If $V$ is any representation of $G$ (resp. $K$) and $\chi$ is a $k$-valued continuous character of $\Q_p^\times$ (resp. $\Z_p^\times$), we will write $V \otimes \chi$ instead of $V \otimes (\chi \circ \det)$.

\subsection*{Acknowledgements}

The author is grateful to Jack Thorne for suggesting the question of studying the $p$-arithmetic cohomology of smooth mod $p$ representations of $\GL_2(\Q_p)$ and for his comments on earlier drafts of this article.

\section{Preliminaries}

\subsection{Arithmetic and \texorpdfstring{$p$}{p}-arithmetic (co)homology}
\label{subsection: cohomology preliminaries}

Let $U^p \subseteq \GL_2(\A^{p \infty})$ be a compact open subgroup of the form $\prod_{\ell \nmid N} \GL_2(\Z_\ell) \times U_N$ for some $N \geq 1$ coprime to $p$ and open compact subgroup $U_N \subseteq \prod_{\ell \mid N} \GL_2(\Q_\ell)$.
Assume $U^p$ is neat.
Let $\GL_2(\R)^\circ$ be the subgroup of $\GL_2(\R)$ consisting of matrices with positive determinant.
The (arithmetic) homology of level $U^p K$ of a left $k[K]$-module $M$, is defined as
$$
    H_i(U^p K, M) := \Tor_i^{k[\GL_2(\Q) \times U^p \times K \times \GL_2(\R)^\circ]} (k[\GL_2(\A)], M).
$$
Here, we view $M$ as a $\GL_2(\Q) \times U^p \times K \times \GL_2(\R)^\circ$ module by letting $\GL_2(\Q) \times U^p \times \GL_2(\R)^\circ$ act trivially, and $k[\GL_2(\A)]$ is acted on by $\GL_2(\Q)$ by multiplication on the left and by $U^p \times K \times \GL_2(\R)^\circ$ by multiplication on the right.
The (arithmetic) cohomology of $M$ in level $U^p K$ is defined analogously as
$$
    H_i(U^p K, M) := \Ext^i_{k[\GL_2(\Q) \times U^p \times K \times \GL_2(\R)^\circ]} (k[\GL_2(\A)], M).
$$
Similarly, if $M$ is a $k[G]$-module, the $p$-arithmetic homology and cohomology of $M$ in level $U^p$ are defined as
\begin{align*}
    H_i(U^p, M) & := \Tor_i^{k[\GL_2(\Q) \times U^p \times \GL_2(\Q_p) \times \GL_2(\R)^\circ]} (k[\GL_2(\A)], M), \\
    H^i(U^p, M) & := \Ext^i_{k[\GL_2(\Q) \times U^p \times \GL_2(\Q_p) \times \GL_2(\R)^\circ]} (k[\GL_2(\A)], M).
\end{align*}
For both arithmetic and $p$-arithmetic homology (and similarly for cohomology), one can canonically define complexes computing them as in \cite[Section 5.1]{S_arithmetic_cohomology}, where they were denoted $C^\ad_\bullet(U^p K, M)$ and $C^\ad_\bullet(U^p, M)$; here we will denote them by $C_\bullet(U^p K, M)$ and $C_\bullet(U^p, M)$ respectively.
One can also speak of arithmetic and $p$-arithmetic hyperhomology and hypercohomology of complexes of $k[K]$ or $k[G]$-modules; these are just the derived tensor products and derived Hom corresponding to the Tor and Ext modules above in their corresponding derived category. 
In this article we will only be interested in the case where
$$
    U^p = U^p_1(N) := \prod_{\ell \nmid N} \GL_2(\Z_\ell) \times \prod_{\ell \mid N} \left\{ g \in \GL_2(\Z_\ell) : g \equiv \begin{pmatrix} * & * \\ 0 & 1 \end{pmatrix} \mod \ell^{v_\ell(N)} \right\}
$$
and $N \geq 3$.
In this case, there are canonical isomorphisms
\begin{align*}
    H_*(U^p_1(N) K, \blank) & \simeq H_*(\Gamma_1(N), \blank),
    &
    H^*(U^p_1(N) K, \blank) & \simeq H^*(\Gamma_1(N), \blank),
    \\
    H_*(U^p_1(N), \blank) & \simeq H_*(\Gamma^p_1(N), \blank),
    &
    H^*(U^p_1(N), \blank) & \simeq H^*(\Gamma^p_1(N), \blank),
\end{align*}
where the right-hand sides denote group homology or cohomology, $\Gamma_1(N) \subseteq \SL_2(\Z)$ is the usual congruence subgroup and $\Gamma^p_1(N)$ is defined as in the introduction.
The arithmetic (resp. $p$-arithmetic) (co)homology groups are non-zero in degrees outside the range $[0,1]$ (resp. $[0,3]$).

\subsection{Hecke operators}
\label{subsection: Hecke operators}

Let $H$ is any locally profinite group, $H_0$ a compact open subgroup and $H_+ \subseteq H$ a submonoid containing $H_0$, and write $H_- = H_+^{-1}$.
If $M$ is a (left) $k[H_-]$-module (resp. $k[H_+]$-module), then the $H_0$-coinvariants $M_{H_0}$ (resp. $H_0$-invariants $M^{H_0}$) are naturally a right (resp. left) module for the Hecke algebra $\Hcal(H_+, H_0)_k$, the algebra of smooth compactly supported $H_0$-biinvariant functions $H_+ \to H$ under convolution.

This discussion applies in particular to arithmetic and $p$-arithmetic (co)homology.
Let $U^p$ and $N$ be as in the previous section.
Let $\Tbb(pN)$ denote the abstract unramified Hecke algebra for $\GL_2$ away from $pN$ with coefficients in $\Z$, that is, the restricted tensor product of the local Hecke algebras $\Hcal(\GL_2(\Q_\ell), \GL_2(\Z_\ell))$ with $\ell \nmid pN$.
It is a commutative algebra freely generated by Hecke operators $T_\ell$, corresponding to the double coset of $\begin{pmatrix} 1 & 0 \\ 0 & \ell \end{pmatrix}$, and invertible operators $S_\ell$, corresponding to the double coset of $\begin{pmatrix} \ell & 0 \\ 0 & \ell \end{pmatrix}$, where $\ell \nmid pN$.
Fix also a submonoid $G_+ \subseteq G$ containing $K$.
Let $V$ be a representation of $G_-$ (resp. $G_+$) on a $k$-vector space.
Then, the arithmetic homology $H_*(U^p K, V)$ (resp. the arithmetic cohomology $H^*(U^p K, V)$) is the $U^p K$-coinvariants (resp. invariants) of a representation of $\GL_2(\A^{p \infty}) \times G_+$, and is thus endowed with commuting actions of $\Tbb(pN)$ and $\Hcal(G_+, K)_k$, the latter being a right (resp. left) action.
For us, $\Hcal(G_+, K)_k$ will always be a commutative algebra so we will not distinguish between left and right actions.
Similarly, if $V$ is a representation of $G$ on a $k$-vector space, then the $p$-arithmetic homology $H_*(U^p, V)$ (resp. the $p$-arithmetic cohomology $H^*(U^p, V)$) is endowed with an action of $\Tbb(pN)$.
If $V$ is a representation of $G_-$ (resp. $G$) and $V^\vee$ denotes its (abstract) contragredient, then there are $\Tbb(pN) \otimes \Hcal(G_+, K)$-equivariant (resp. $\Tbb(pN)$-equivariant) isomorphisms $H^*(U^p K, V^\vee) \simeq H_*(U^p K, V)^\vee$ (resp. $H^*(U^p, V^\vee) \simeq H_*(U^p, V)^\vee$).
In particular, the systems of eigenvalues in both spaces are the same, so all the statements for cohomology in \Cref{thm: 1} follow from the corresponding statements for homology.
For this reason, we will now work almost exclusively with homology.

Let $\rho$ be a $k$-valued continuous representation of $\Gal(\bar \Q / \Q)$ unramified outside $N$, and consider the maximal ideal $\m_\rho$ of $\Tbb(pN)$ defined by the kernel of the homomorphism $\Tbb(pN) \to k$ defined by $T_\ell \mapsto \tr \rho(\Frob_\ell)$ and $\ell S_\ell \mapsto \det \rho(\Frob_\ell)$ for $\ell \nmid p$, where $\Frob_\ell$ is a geometric Frobenius at $\ell$.
For example, $\m_{1 \oplus \varepsilon^{-1}}$ is generated by $T_\ell - (1 + \ell), \ell S_\ell - \ell$.
Given a $\Tbb(pN)$-module $M$, we will say that $\rho$ contributes to, or appears in, $M$ is the localisation of $M_{\m_\rho}$ is non-zero.
We will sometimes write $M_\rho$ instead of $M_{\m_\rho}$.
If $V$ is an irreducible representation of $K$, then any system of Hecke eigenvalues in $H^*(\Gamma_1(N), V)$ corresponds to a semisimple 2-dimensional Galois representation as above.
For such a $V$, the cohomology $H^0(\Gamma_1(N), V)$ vanishes unless $V$ is the trivial representation, in which case it is one-dimensional and the Hecke operators act via $T_\ell = 1 + \ell, S_\ell = 1$.
In particular, the only semisimple Galois representation $\rho$ that can contribute to $H^0(\Gamma_1(N), V)$ is $1 \oplus \varepsilon^{-1}$.
In particular, irreducible Galois representations can only contribute to arithmetic cohomology in level $\Gamma_1(N)$ only in degree 1.

Given a character $\chi \colon \Q_p^\times \to k^\times$, write $k(\chi)$ for the $\Tbb(pN)$-module whose underlying module is $k$ and where $T_\ell$ (resp. $S_\ell$) act via $\chi(\ell)$ (resp. $\chi(\ell)^2$).
For any $\Tbb(pN)$-module $M$, write $M(\chi) := M \otimes_k k(\chi)$ for the twist of $M$ by $k(\chi)$.

\subsection{Irreducible mod \texorpdfstring{$p$}{p} representations of \texorpdfstring{$\GL_2(\Q_p)$}{GL\_2(Q\_p)}}
\label{subsection: mod p preliminaries}

In this section we will recall the construction of the smooth irreducible mod $p$ representations of $G$ and some facts about them.
Given $0 \leq r \leq p-1$, consider the representation $\Sym^r (k^2)^\vee$ of $K$ over $k$.
Note that $\Sym^r(k^2)$ naturally extends to a representation of the monoid $G^+$ of matrices with entries in $\Z_p$, and also (perhaps less naturally) to a representation of $KZ$ where $\beta$ acts trivially.
In particular, $\Sym^r (k^2)^\vee$ extends to representations of $G^-$ and of $KZ$.
The first action defines an action of $\Hcal(G^+, K)$ on the compact induction $\cInd_K^G(\Sym^r (k^2)^\vee)$ and the second defines an action of $\Hcal(KZ, K)$.
We let $T$ denote the operator on $\cInd_K^G(\Sym^r (k^2)^\vee)$ corresponding to the double coset of $\alpha$ under the first action and $\diamondp$ the operator corresponding to the double coset of $\beta$ under the second, which is an invertible operator.
Explicitly, these operators are described as follows.
Given $g \in G$ and $v \in \Sym^r (k^2)^\vee)$, let $[g, v]$ denote the element of $\cInd_K^G(\Sym^r (k^2)^\vee)$ that is supported on $gK$ and maps $g$ to $v$.
We will use the same notation for compact inductions from other groups and of other representations.
The Hecke operators above are then defined by the formulas
\begin{align*}
    T [g, v] & = \sum_{x \in K / I} [g x \alpha, \alpha^{-1} x^{-1} v],
    &
    S [g, v] & = [\beta g, v],
\end{align*}

Given a continuous character $\chi \colon \Q_p^\times \to k^\times$ and $\lambda \in k$, define
$$
    \pi(r, \lambda, \chi) := \frac{\cInd_K^G (\Sym^r (k^2)^\vee)}{(T - \lambda, \diamondp - 1)} \otimes_k \chi \omega^r
$$
(we have included the twist by $\omega^r$ in order for this notation to match that of the existing literature).

\begin{remark}\label{remark: twist changes eigenvalues}
    We will later need to consider a variation of this definition that is defined in families.
    Let $R$ be a $k$-algebra.
    One can define Hecke operators $T$ and $S$ on $\cInd_K^G (\Sym^r (k^2)^\vee \otimes_k \chi)$ for any character $\chi \colon G \to R^\times$ in the same way as for $R = k$ and $\chi = 1$, and the natural isomorphism $\cInd_K^G (\Sym^r (k^2)^\vee \otimes_k \chi) \simto \cInd_K^G (\Sym^r (k^2)^\vee) \otimes_k \chi$ intertwines $T$ and $S$ on the source with $T \otimes 1$ and $S \otimes 1$ respectively on the target.
    Moreover, if $b \in R^\times$ and $\lambda \in R$, then there are isomorphisms of representations of $G$
    \begin{align*}
        \frac{\cInd_K^G (\Sym^r (k^2)^\vee \otimes_k R)}{(T - \lambda, S - 1)} \otimes_R (\chi \unr{b})
        & \simto
        \frac{\cInd_K^G (\Sym^r (k^2)^\vee \otimes_k (\chi \unr{b}))}{(T - \lambda, S - 1)}
        \\ & \simto
        \frac{\cInd_K^G (\Sym^r (k^2)^\vee \otimes_k \chi)}{(T - \lambda b, S - b^2)}.
    \end{align*}
    Let us define for $\tau \in R$, $\sigma \in S$ and $s \in \Z$,
    $$
        \tilde \pi(r, \tau, \sigma, s)_R := \frac{\cInd_K^G (\Sym^r (k^2)^\vee \otimes_k \omega^{s} \otimes_k R)}{(T - \tau, S - \sigma)}.
    $$
    In particular, when $R = k$, then $\pi(r, \lambda, \omega^a \unr{b}) \simeq \tilde \pi(r, \tau, \sigma, s)_k$ for $\tau = \lambda b, \sigma = b^2$ and $s = a + r$.
\end{remark}

The following results are due to Barthel--Livn\'e \cite{barthel_livne} and Breuil \cite{breuil_sur_quelques}.

\begin{theorem}\label{thm: reminder of mod p representations}
    \begin{enumerate}
    \fixitem If $(r, \lambda) \neq (0, \pm 1), (p-1, \pm 1)$, then $\pi(r, \lambda, \chi)$ is irreducible.
    \item\label{item: isomorphism to principal series} If $\lambda \neq 0$ and $(r, \lambda) \neq (0, \pm 1)$ then there is a canonical isomorphism
    $$
        \pi(r, \lambda, \chi) \simto \Ind_B^G (\chi \unr{\lambda^{-1}} \otimes \chi \unr{\lambda} \omega^r).
    $$
    Here, $\Ind_B^G(\blank)$ denotes smooth parabolic induction.
    The representation $\pi(r, \lambda, \chi)$ is said to be a principal series representation.
    \item If $r = 0$ and $\lambda =\pm 1$ there is a canonical homomorphism
    $$
        \pi(r, \lambda, \chi) \to \Ind_B^G (\chi \unr{\lambda^{-1}} \otimes \chi \unr{\lambda})
    $$
    with kernel and cokernel isomorphic to $\St \otimes \chi \unr{\lambda}$ and image isomorphic to $\chi \unr{\lambda} \circ \det$.
    Here, $\St$ is the Steinberg representation of $G$ over $k$, i.e. the quotient of $\Ind_B^G(k)$ by the trivial representation $k$.
    \item If $\lambda = 0$, then $\pi(r, 0, \chi)$ is not isomorphic to a principal series representation; it is said to be supersingular.
    \item Any irreducible representation of $G$ over $k$ is isomorphic to a principal series, a supersingular, a twist of the Steinberg, or a twist of the trivial representation.
    \item\label{item: isomorphisms pis ordinary} If $\lambda \neq 0$, the only isomorphisms between various of these representations are
    $$
        \pi(r, \lambda, \chi) \simeq \pi(r, -\lambda, \chi \unr{-1})
    $$
    and for $\lambda \neq \pm 1$
    $$
        \pi(0, \lambda, \chi) \simeq \pi(p-1, \lambda, \chi).
    $$
    \item If $\lambda = 0$, the isomorphisms between various of these representations are
    \begin{align*}
        \pi(r, 0, \chi)
        & \simeq \pi(r, 0, \chi \unr{-1}) \\
        & \simeq \pi(p-1-r, 0, \chi \omega^r) \\
        & \simeq \pi(p-1-r, 0, \chi \omega^r \unr{-1}).
    \end{align*}
\end{enumerate}
\end{theorem}

In particular, there is a non-zero map, unique up to scalar, $\pi(p-1, 1, \chi) \to \pi(0, 1, \chi)$ (resp. $\pi(0, 1, \chi) \to \pi(p-1, 1, \chi)$), which factors through a twist of the Steinberg (resp. trivial) representation.
In \Cref{section: Hasse invariant}, we will describe these maps explicitly.

\subsection{The mod \texorpdfstring{$p$}{p} Langlands correspondence for \texorpdfstring{$\GL_2(\Q_p)$}{GL\_2(Q\_p)}}

Throughout this section, we assume $p \geq 5$.
We will use the same conventions for the mod $p$ local Langlands correspondence as in \cite{emerton_local_global}, however we also include twists of extensions of the cyclotomic character by the trivial character in our discussion.
Let $\MF$ be Colmez's magical functor, defined by $\MF(\pi) := \mathbf V(\pi) \otimes \omega$ for $\mathbf V$ as in \cite{colmez_representations}.

\begin{theorem}\label{thm: mod p langlands}
    Let $\rho \colon \Gal(\bar \Q_p / \Q_p) \to \GL_2(k)$ be a continuous representation.
    Then, there exists a finite length smooth representation of $G$ over $k$, unique up to isomorphism, satisfying the following properties:
    \begin{enumerate}
        \item $\MF(\pi) \simeq \rho$,
        \item $\pi$ has central character corresponding to $(\det \rho) \omega$ under local class field theory,
        \item $\pi$ has no finite-dimensional $G$-invariant subrepresentations or quotients.
    \end{enumerate}
    More specifically, $\pi$ can be described as follows:
    \begin{enumerate}
        \item If $\rho$ is irreducible, say $\rho = \Ind(\omega_2^{r+1}) \otimes \chi$ with $0 \leq r \leq p-1$, then $\pi = \pi(r, 0, \chi \omega)$ is supersingular.
        
        \item If $\rho \simeq \begin{pmatrix} \chi_1 & * \\ 0 & \chi_2 \end{pmatrix}$ with $\chi_1 \neq \chi_2 \omega^{\pm1}$, then $\pi$ is an extension
        $$
            0 \to \pi(r, \lambda, \chi) \to \pi \to \pi([p-3-r], \lambda^{-1}, \chi \omega^{r+1}) \to 0,
        $$
        where
        \begin{align*}
            \chi_1 & = \omega^r \chi \unr{\lambda}, &
            \chi_2 & = \omega^{-1} \chi \unr{\lambda^{-1}}
        \end{align*}
        with $0 \leq r \leq p-1$, and $[p-3-r]$ is the unique integer between $0$ and $p-2$ that is congruent to $p-3-r$ modulo $p-1$.
        This extension is split if and only if $\rho$ is semisimple.

        \item If $\rho$ is a non-split extension $\begin{pmatrix} \chi & * \\ 0 & \chi \omega^{-1} \end{pmatrix}$,
        then $\pi$ has a unique Jordan--H\"older series, which is of the form
        $$
            0 \subseteq \pi_1 \subseteq \pi_2 \subseteq \pi
        $$
        where $\pi_1 \simeq \St \otimes \chi, \pi_2 / \pi_1 \simeq \chi \circ \det$ and $\pi / \pi_2 \simeq \pi(p-3, 1, \chi \omega)$.
        \item\label{item: llc case bad} If $\rho$ is an extension $\begin{pmatrix} \chi \omega^{-1} & * \\ 0 & \chi \end{pmatrix}$,
        then $\pi$ is an extension
        $$
            0 \to \pi(p-3, 1, \chi \omega) \to \pi \to \St \otimes \chi \to 0.
        $$
        This extension is split if and only if $\rho$ is semisimple.
        On both the Galois side and the $\GL_2(\Q_p)$ side, there is a unique class of non-trivial extensions, so this property determines $\pi$.
    \end{enumerate}
\end{theorem}
\begin{proof}
    All the statements follow from the work of Colmez \cite[Section VII.4]{colmez_representations} except for case \Cref{item: llc case bad}, which follows from the end of the proof \cite[Lemma 10.35]{paskunas_image} by taking into account that there is only one isomorphism class of Galois representations that are non-split extensions of 1 by $\omega^{-1}$.
\end{proof}

For $\rho$ as in the theorem, we will say that $\pi$ is the representation corresponding to $\rho$ under the mod $p$ local Langlands correspondence.
One could argue (for example, following \cite[Remark 7.7]{patching_gl2}) that the ``true" mod $p$ local Langlands correspondence in case \Cref{item: llc case bad} should be (up to isomorphism) a non-trivial extension of $\pi$ as above by two copies of the trivial character.
However, we are only interested in the socle of $\pi$, which remains the same and can be easily described in general by the following result.

\begin{proposition}\label{proposition: socle llc generic reducible}
    Let $\rho$ be a representation $\Gal(\bar \Q_p / \Q_p) \to \GL_2(k)$ and let $\pi$ be the corresponding representation of $\GL_2(\Q_p)$. Then, the following statements hold.
    \begin{enumerate}
        \item The representation $\pi(r, 0, \chi)$ is a subrepresentation of $\pi$ if and only if $\rho \simeq \Ind(\omega_2^{r + 1}) \otimes \chi \omega^{-1}$.
        \item For $\lambda \neq 0$ and $(r, \lambda) \neq (0, \pm 1), (p-1, \pm 1)$, $\pi(r, \lambda, \chi)$ is a subrepresentation of $\pi$ if and only if
        $$
            \rho \simeq \begin{pmatrix}
                \omega^r \unr{\lambda} & * \\
                0 & \omega^{-1} \unr{\lambda^{-1}}
            \end{pmatrix} \otimes \chi.
        $$
        \item $\St \otimes \chi$ is a subrepresentation of $\pi$ if and only if
        $$
            \rho \simeq \begin{pmatrix}
                1 & * \\
                0 & \omega^{-1}
            \end{pmatrix} \otimes \chi.
        $$
        \item $\chi \circ \det$ is never a subrepresentation of $\pi$.
    \end{enumerate}
\end{proposition}
\begin{proof}
    This follows from \Cref{thm: mod p langlands} and \Cref{thm: reminder of mod p representations} \Cref{item: isomorphisms pis ordinary}.
\end{proof}

\subsection{The weight part of Serre's conjecture}
\label{subsection: serres conjecture}

In this section we will recall the answer to the following question: given a continuous representation $\rho \colon \Gal(\bar \Q / \Q) \to \GL_2(k)$, for what $0 \leq r \leq p-1$ and $s$ does $\rho$ contribute to $H^1(\Gamma_1(N), \Sym^r (k^2) \otimes \omega^{- s})$ and, when it does, what are its eigenvalues for the Hecke operators at $p$?
What we mean by Hecke operators at $p$ is the following.
One can define the action of Hecke operators $T$ and $S$ on the arithmetic homology complex $C_\bullet(U^p K, \Sym^r (k^2)^\vee \otimes_k \omega^{s})$ for any tame level $U^p \subseteq \GL_2(\A^{p \infty})$
as we explained in \Cref{subsection: Hecke operators}
by using the actions of $G^-$ and $KZ$ on $\Sym^r (k^2)^\vee \otimes_k \omega^{s})$, and similarly for the cohomology of the dual.
When $s = 0$, these operators correspond under the Eichler--Shimura isomorphism followed by reduction mod $p$ to the usual Hecke operators $T_p$ and $\langle p \rangle$ respectively for modular forms of level $\Gamma_1(N)$ and weight $r+2$.

The answer to the question is contained in the proof of \cite[Theorem 3.17]{buzzard_diamond_jarvis}.
We warn the reader that our conventions are different to those in \cite{buzzard_diamond_jarvis}: $\rho$ contributes to the cohomology of a Serre weight $V$ in the sense of this article if and only if $\rho^\vee$ is modular of weight $V$ in the sense of \cite{buzzard_diamond_jarvis} (equivalently, if and only if $\rho$ is modular of weight $V^\vee \otimes \omega^{-1}$ in the sense of \cite{buzzard_diamond_jarvis}).
In particular, it follows from \cite[Corollary 2.11]{buzzard_diamond_jarvis} that $\rho$ contributes to the cohomology of $V \otimes \omega^a$ if and only if $\rho \varepsilon^a$ contributes to the cohomology of $V$ for any $a$.

\begin{theorem}
\label{thm: serre weight conjecture}
    Let $\rho \colon \Gal(\bar \Q / \Q) \to \GL_2(k)$ be an odd irreducible Galois representation.
    Let $0 \leq r \leq p-1$ and $a \in \Z$.
    Let $\lambda \in k, b \in k^\times$ and set $\tau = \lambda b$, $\sigma = b^2$ and $s = a + r$.
    Then, $\rho$ contributes to the $(T = \tau, S = \sigma)$-eigenspace in $H^1(\Gamma_1(N), \Sym^r (k^2) \otimes \omega^{-s})$ if and only if $N(\rho)$ divides $N$ and one of the following holds:
    \begin{enumerate}
        \item $\lambda = 0$ and $\rho|_{\Gcal_p} \simeq \Ind(\omega_2^{r+1}) \otimes \omega^{a-1
} \unr{b}$,
        \item $\lambda \neq 0$, $(r, \lambda) \neq (0, \pm 1)$ and
        $$
            \rho|_{\Gcal_p} \simeq \begin{pmatrix}
                \omega^r \unr{\lambda} & * \\
                0 & \omega^{-1} \unr{\lambda^{-1}}
            \end{pmatrix} \otimes \omega^{a} \unr{b}.
        $$
        \item $r = 0$, $\lambda = \pm 1$ and
        $$
            \rho|_{\Gcal_p} \simeq \begin{pmatrix}
                1 & * \\
                0 & \omega^{-1}
            \end{pmatrix} \otimes \omega^{a} \unr{\lambda b}.
        $$
        where $*$ denotes a peu ramifi\'ee extension.
    \end{enumerate}
\end{theorem}
\begin{proof}
    According to \cite[Theorem 2.5 and Theorem 2.6]{edixhoven_weight}, only $\rho$ whose restriction to $\Gcal_p$ is irreducible (resp. reducible) can contribute to the eigenspaces where $T = 0$ (resp. $T \neq 0$).
    By \cite[Theorem 2.6]{edixhoven_weight}, the representations $\rho$ appearing in the $(T = 0, S = b^2)$-eigenspace of $H^1(\Gamma_1(N), \Sym^{r} (k^2) \otimes \omega^{-s})$ are those satisfying $\rho|_{\mathcal{I}_p} \omega^{-s} \simeq (\omega_2^{r+1} \oplus \omega_2^{p(r+1)}) \otimes \omega^{-r-1}$ and $(\det \rho)(\Frob_p) = b^2$ (recall that $\Frob_p$ is a Frobenius mapping to $p$ under class field theory, so it is a well-defined element of $\Gcal_p^{\mathrm{ab}}$).
    In particular, the restriction to $\Gcal_p$ of such a representation has determinant $\omega^{2 s-r-1} \unr{b^2}$, so it must be isomorphic to
    $$
        \Ind(\omega_2^{r+1}) \otimes \omega^{s-r-1} \unr{b}
        \simeq
        \Ind(\omega_2^{r+1}) \otimes \omega^{a-1} \unr{b}.
    $$

    For the case where $\lambda \neq 0$, \cite[Theorem 2.5]{edixhoven_weight} states that, given a system of Hecke eigenvalues in the $(T = \tau, S = \sigma)$-eigenspace of $H^1(\Gamma_1(N), \Sym^{r} (k^2) \otimes \omega^{-s})$, then
    $$
        \rho|_{\Gcal_p} \omega^{-s} \simeq \begin{pmatrix}
            \unr{\tau} & * \\
            0 & \omega^{-r-1} \unr{\tau^{-1} \sigma}
        \end{pmatrix}.
    $$
    If $\tau = \lambda b$ and $\sigma = b^2$, this is equivalent to
    $$
        \rho|_{\Gcal_p} \simeq \begin{pmatrix}
            \omega^r \unr{\lambda} & * \\
            0 & \omega^{-1} \unr{\lambda^{-1}}
        \end{pmatrix} \otimes \omega^{a} \unr{b}.
    $$
    Conversely, \cite[Theorem 3.17]{buzzard_diamond_jarvis} and its proof show that a representation of this form does contribute to $H^1(\Gamma_1(N), \Sym^{r} (k^2) \otimes \omega^{-s})$ provided that the extension is peu ramifi\'ee whenever $r = 0$ and $\lambda = \pm 1$ (and, in this case, it never contributes if the extension is tr\`es ramifi\'ee).
    It remains to show that it contributes to the $(T = \lambda b, S = b^2)$-eigenspace.
    When $r \neq p-2$ or $\rho|_{\Gcal_p}$ is non-split, \cite[Theorems 2.5 and 2.6]{edixhoven_weight} again show that this the only eigenspace for Hecke operators at $p$ to which $\rho$ can contribute, so it must do so.
    In the case when $r = p-2$ and $\rho|_{\Gcal_p} \simeq \omega^{a-1} \unr{\lambda b} \oplus \omega^{a-1} \unr{\lambda^{-1} b}$, the same results show that $\rho$ can only appear in the eigenspaces for $(T = \lambda b, S = b^2)$ and $(T = \lambda^{-1} b, S = b^2)$.
    We know that $\rho$ contributes to at least one of these, and we must show that $\rho$ contributes to both.
    When $\lambda = \pm 1$ this is clear, since both systems of eigenvalues are actually the same.
    When $\lambda \neq \pm 1$, this follows from \cite[Theorem 13.10]{gross_tameness}.
\end{proof}

\section{\texorpdfstring{$p$}{p}-arithmetic homology of \texorpdfstring{$\pi(r, \lambda, \chi)$}{pi(r, l, x)}}

\subsection{\texorpdfstring{$p$}{p}-arithmetic and arithmetic homology}
\label{subsection: parithmetic and arithmetic}

In this section we will relate the $p$-arithmetic homology of the representations $\pi(r, \lambda, \chi)$ to the arithmetic homology of Serre weights $\Sym^r (k^2)^\vee \otimes \omega^{s}$.
The argument is the same as that of \cite{S_arithmetic_cohomology}.
As in \Cref{remark: twist changes eigenvalues}, $R$ is a $k$-algebra.

\begin{lemma}
For any $0 \leq r \leq p-1$ and $s \in \Z$, $\tau \in R$ and $\sigma \in R^\times$, we have
$$
    \Tor^i_{R[T, \diamondp]}(R[T, \diamondp]/(T - \tau, \diamondp-\sigma ), \cInd_K^G (\Sym^r (k^2)^\vee \otimes_k \omega^{s} \otimes_k R)) = 0
$$
for $i > 0$.
\end{lemma}
\begin{proof}
This follows from the fact that $(\diamondp - \sigma, T - \tau)$ is a regular sequence for the $k[T,S]$-module $\cInd_K^G (\Sym^r (k^2)^\vee \otimes_k \omega^{s} \otimes_k R)$, which can be seen by studying how $T$ and $S$ modify the support of a function using the Cartan decomposition.
See \cite[Lemma 4.10]{patching_gl2} for the details.
\end{proof}

This shows that the $G$-module $\tilde \pi(r, \tau, \sigma, s)_R$ from \Cref{remark: twist changes eigenvalues} can be written not just as the eigenquotient
\begin{align*}
    \cInd_K^G & (\Sym^r (k^2)^\vee \otimes_k \omega^{s} \otimes_k R) / (T - \tau, \diamondp - \sigma) \\
    & \simeq
    \frac{R[T, S]}{(T - \tau, \diamondp - \sigma)} \otimes_{R[T, S]} \cInd_K^G(\Sym^r (k^2)^\vee \otimes_k \omega^{s} \otimes_k R)
\end{align*}
but also as a \emph{derived} eigenquotient: there is an isomorphism in the derived category of (abstract\footnote{In our arguments involving homological algebra, we will always work with categories of abstract representations (of $G$ or other groups) and never with categories of smooth representations}) $R[G][T,S]$-modules
$$
    \tilde \pi(r, \tau, \sigma, s)_R \simeq \frac{R[G][T, S]}{(T - \tau, \diamondp - \sigma)} \otimesL_{R[G][T, S]} \cInd_K^G(\Sym^r (k^2)^\vee \otimes_k \omega^{s} \otimes_k R).
$$
Fix $U^p$ and $N$ as in \Cref{subsection: cohomology preliminaries}.
We can define an action of Hecke operators $T$ and $S$ in arithmetic homology over $R$ in the same way as described in the beginning of \Cref{subsection: serres conjecture},
and the arguments in \cite[Section 5.8]{S_arithmetic_cohomology} show that we have an isomorphism in the derived category of $\Tbb(pN) \otimes_\Z R[T, S]$-modules for the $p$-arithmetic homology complex
\begin{align*}
    C_\bullet(U^p, \tilde \pi(r, \tau, \sigma, s)_R)
    & \simeq
    C_\bullet(U^p K, \Sym^r (k^2)^\vee \otimes_k \omega^{s} \otimes_k R) \otimesL_{R[T, S]} \frac{R[T, S]}{(T - \tau, \diamondp - \sigma)}.
\end{align*}
Moreover,
\begin{align*}
    C_\bullet(U^p K, \Sym^r (k^2)^\vee \otimes_k \omega^{s} \otimes_k R) \simeq C_\bullet(U^p K, \Sym^r (k^2)^\vee \otimes_k \omega^{s}) \otimesL_k R.
\end{align*}
Thus, in fact
\begin{align*}
    C_\bullet(U^p, \tilde \pi(r, \tau, \sigma, s)_R)
    &
    \simeq
    C_\bullet(U^p K, \Sym^r (k^2)^\vee \otimes_k \omega^{s}) \otimesL_{k[T,S]} \frac{R[T, S]}{(T - \tau, S - \sigma)}.
\end{align*}

\begin{remark}\label{remark: reason for families}
    The reason why we have considered representations in families is the following.
    Assume that $R = k[\tau, \sigma, \sigma^{-1}]$ for two indeterminate variables $\tau$ and $\sigma$.
    Then,
    \begin{align*}
        C_\bullet(U^p, \tilde \pi(r, \tau, \sigma, s)_R)
        &
        \simeq
        C_\bullet(U^p K, \Sym^r (k^2)^\vee \otimes_k \omega^{s}) \otimesL_{k[T,S]} \frac{R[T, S]}{(T - \tau, S - \sigma)}
        \\ &
        \simeq
        C_\bullet(U^p K, \Sym^r (k^2)^\vee \otimes_k \omega^{s}) \otimesL_{k[T,S,S^{-1}]} \frac{R[T, S,S^{-1}]}{(T - \tau, S - \sigma)}
        \\ &
        \simeq
        C_\bullet(U^p K, \Sym^r (k^2)^\vee \otimes_k \omega^{s}) \otimesL_{k[T,S, S^{-1}]} \frac{k[T, S, S^{-1}, \tau, \sigma, \sigma^{-1}]}{(T - \tau, S - \sigma)}
        \\ &
        \simeq
        C_\bullet(U^p K, \Sym^r (k^2)^\vee \otimes_k \omega^{s}),
    \end{align*}
    where the last term is viewed as an $R$-module by letting $\tau$ act as $T$ and $\sigma$ as $S$.
    In other words, the $p$-arithmetic homology over $R$ coincides with the corresponding arithmetic homology, and not just a (derived) eigenquotient of it.
\end{remark}

\begin{proposition}
\label{proposition: spectral sequence}
    There is a spectral sequence converging to $H_*(U^p, \tilde \pi(r, \tau, \sigma, s)_R)$
    whose $E^2$ page is
    $$
        E^2_{i,j} = \Tor_i^{k[T,S]} \left( \frac{R[T, S]}{(T - \tau, \diamondp - \sigma)}, H_j(U^p K, \Sym^r (k^2)^\vee \otimes_k \omega^{s}) \right).
    $$
    In particular, there is a spectral sequence converging to $H_*(U^p, \pi(r, \lambda, \omega^a \unr{b}))$
    whose $E^2$ page is
    $$
        E^2_{i,j} = \Tor_i^{k[T,S]} \left( \frac{k[T, S]}{(T - \lambda b, \diamondp - b^2)}, H_j(U^p K, \Sym^r (k^2)^\vee \otimes_k \omega^{a+r}) \right).
    $$
\end{proposition}

\begin{remark}\label{remark: dimension of Tors}
    When $R = k$, the $k[T,S]$-module resolution
    $$
        k[T,S] \xrightarrow{(S - \sigma) \oplus (\tau - T)} k[T,S]^2 \xrightarrow{(T - \tau, S - \sigma)} k[T,S]
    $$
    of $k[T, S] / (T - \tau, \diamondp - \sigma)$ shows that for any $k[T,S]$-module $V$ that is finite-dimensional as a $k$-vector space, $\Tor^{k[T,S]}_i(k[T, S] / (T - \tau, \diamondp - \sigma), V)$ vanishes in degrees outside the range $[0, 2]$, and is isomorphic in degree 2 (resp. degree 0) to the $(T = \tau, S = \sigma)$-eigenspace (resp. eigenquotient) of $V$.
    Moreover, the Tor modules in degree 1 lie in a short exact sequence
    \begin{align*}
        0 & \to
        \frac{k[T]}{(T - \tau)} \otimes_{k[T]} \Hom_{k[S]} \left( \frac{k[S]}{(S - \sigma)}, V \right)
        \\ & \to
        \Tor^{k[T,S]}_1 \left( \frac{k[T, S]}{(T - \tau, \diamondp - \sigma)}, V \right)
        \\ & \to
        \Hom_{k[T]} \left( \frac{k[T]}{(T - \tau)}, \frac{k[S]}{(S - \sigma)} \otimes_{k[S]} V \right)
        \to 0.
    \end{align*}
    In particular, the Tor groups vanish if and only if they vanish in at least one of the degrees 0, 1 or 2.
    The dimensions $d_0$ and $d_2$ of the Tor spaces in degrees 0 and 2 are equal, and the dimension in degree 1 is $2 d_0 = 2 d_2$.
\end{remark}

\subsection{Proof of \texorpdfstring{\Cref{thm: 1}}{Theorem 1.1} in the generic case}
\label{subsection: proof generic}

Parts \Cref{item: thm 1 i} and \Cref{item: thm 1 ii} of \Cref{thm: 1} follow immediately from \Cref{proposition: spectral sequence} for supersingular and principal series representations, as well as their analogue for the reducible representations $\pi(0, \pm 1, \chi)$ and $\pi(p-1, \pm1, \chi)$ (by the corresponding results for arithmetic homology).

Moreover, if $\rho$ is an odd irreducible 2-dimensional Galois representations, then the localisation at $\rho$ of the spectral sequence from \Cref{proposition: spectral sequence} satisfies $(E^2_{i,j})_\rho = 0$ for $j \neq 1$, so we may conclude that
$$
    H_{i+1}(\Gamma^p_1(N), \pi(r, \lambda, \omega^a \unr{b}))_\rho \simeq \Tor_i^{k[T,S]} \left( \frac{k[T, S]}{(T - \lambda b, \diamondp - b^2)}, H_1(U^p K, \Sym^r (k^2)^\vee \otimes_k \omega^{a+r})_\rho \right).
$$
In particular, taking into account \Cref{remark: dimension of Tors}, the following are equivalent:
\begin{enumerate}
    \item $\rho$ contributes to the $p$-arithmetic homology $H_*(\Gamma^p_1(N), \pi(r, \lambda, \omega^a \unr{b}))$, and it does exactly in degrees 1, 2 and 3,
    \item $\rho$ contributes to the degree 1 $p$-arithmetic homology $H_1(\Gamma^p_1(N), \pi(r, \lambda, \omega^a \unr{b}))$,
    \item $\rho$ contributes to the $p$-arithmetic homology $H_*(\Gamma^p_1(N), \pi(r, \lambda, \omega^a \unr{b}))$,
    \item $\rho$ contributes to the $(T = \lambda b, S = b^2)$-eigenspace of $H_1(\Gamma_1(N), \Sym^r (k^2)^\vee \otimes \omega^{a+r})$,
    \item $\rho$ contributes to the $(T = \lambda b, S = b^2)$-eigenspace of $H^1(\Gamma_1(N), \Sym^r (k^2) \otimes \omega^{-a-r})$.
\end{enumerate}
By \Cref{thm: serre weight conjecture} and \Cref{proposition: socle llc generic reducible}, when $\pi(r, \lambda, \omega^a \unr{b})$ is irreducible, these are equivalent to $N(\rho)$ dividing $N$ and this representation appearing in the socle of the smooth representation of $\GL_2(\Q_p)$ associated to $\rho|_{\Gcal_p}$ by the mod $p$ local Langlands correspondence of \Cref{thm: mod p langlands}, which proves part \Cref{item: thm 1 iii} of \Cref{thm: 1} in this case.
For later reference, we also record the following proposition, which follows in the same way from \Cref{proposition: spectral sequence} and \Cref{thm: serre weight conjecture}.

\begin{proposition}\label{proposition: representations in weights 0 and p-1}
    Let $\rho \colon \Gal(\bar \Q / \Q) \to \GL_2(k)$ be an odd irreducible representation. Then, the space $H_1(\Gamma^p_1(N), \pi(p-1, 1, \chi))$ (resp. to $H_1(\Gamma^p_1(N), \pi(0, 1, \chi))$) is finite-dimensional, and any system of Hecke eigenvalues in it has an attached Galois representation. Moreover, $\rho$ contributes to this space if and only if $N(\rho)$ divides $N$, and $\rho|_{\Gcal_p}$ is isomorphic to an extension (resp. a peu ramifi\'ee extension)
    $$
        \begin{pmatrix}
            1 & * \\
            0 & \omega^{-1}
        \end{pmatrix}
        \otimes \chi.
    $$
\end{proposition}

\section{Preparation for the non-generic cases}
\label{section: Hasse invariant}

In order to deal with the Steinberg and trivial cases, we will need a few preliminaries on the non-zero maps $\pi(p-1,1,\chi) \to \pi(0,1,\chi)$ and $\pi(0,1,\chi) \to \pi(p-1,1,\chi)$ and the corresponding maps on $p$-arithmetic homology.
They turn out to be related to the degeneracy maps from modular forms of level $\Gamma_1(N)$ to level $\Gamma_1(N) \cap \Gamma_0(p)$ induced by $\tau \mapsto \tau$ and $\tau \mapsto p \tau$ and to
the group cohomological avatar of multiplication by the Hasse invariant studied by Edixhoven--Khare in \cite{edixhoven_khare_hasse}.
Our next goal is to study these maps.

\subsection{The map \texorpdfstring{$\pi(p-1,1,1) \to \pi(0,1,1)$}{pi(p-1,1,1) -> pi(0,1,1)}}
\label{subsection: map pi 1}

Recall from \Cref{subsection: mod p preliminaries} that there is a unique-up-to-scalaras non-zero map $\pi(p-1,1,\chi) \to \pi(0,1,\chi)$.
The goal of this section is to give an explicit description of this map.

We may assume that $\chi = 1$.
First, let us observe that there is a $K$-module isomorphism $k \oplus \Sym^{p-1} (k^2)^\vee \simto \Map(\P^1(\F_p), k)$, which identifies the trivial representation with the subrepresentation of constant functions $\P^1(\F_p) \to k$ and $\Sym^{p-1} (k^2)^\vee$ with the subrepresentation of functions whose total sum equals 0.
As usual, one can also identify $\Sym^{p-1} (k^2)^\vee$ with the space of homogeneous polynomial functions of degree $p-1$ in two variables.
The former identification is then given by sending a homogeneous polynomial function $Q$ of degree $p-1$ in two variables to the function $(x : y) \mapsto Q(x, y)$.

In fact, this can be upgraded to a $G^-$-equivariant isomorphism in a natural way.
The action of the monoid $G^+$ on $\F_p^2$ descends to an action on $\F_p^2 / \F_p^\times = \P^1(\F_p) \cup \{ 0 \}$.
It is easy to check (for example, using the Cartan decomposition of $G$) that an element $g \in G^+$ can act in three ways: invertibly (if $g \in K$), by sending everything to 0 (if all the entries of $g$ are multiples of $p$), or by mapping 0 and one point of $\P^1(\F_p)$ to 0 and all other points of $\P^1(\F_p)$ to another (fixed) point of $\P^1(\F_p)$.
Thus, we get an action of $G^-$ on $\Map(\P^1(\F_p) \cup \{ 0 \}, k)$, and it follows from the previous sentence that the subspace of functions such that $f(0) = \sum_{P \in \P^1(\F_p)} f(P)$ is stable under this action.
This space can be naturally identified with $\Map(\P^1(\F_p), k)$ by restriction, and the resulting action of $G^-$ on this space makes the isomorphisms in the previous paragraph $G^-$-equivariant.
Naturally, these isomorphisms are also $KZ$-equivariant when we instead extend the action of $K$ to one of $KZ$ by letting $\beta$ act trivially.
There is a $K$-equivariant isomorphism $\Map(\P^1(\F_p), k) \to \cInd_I^K (k)$ given by sending a function $f$ to $\begin{pmatrix} a & b \\ c & d \end{pmatrix} \mapsto f(a : c)$, and we will view the target as a $G^-$-module and a $KZ$-module (whose underlying $K$-module structures agree) by transport of structure.
In particular, compactly inducing to $G$ we obtain actions of Hecke operators $T$ and $S$ as usual, and the maps above induce $k[T, S]$-module isomorphisms
$$
    \cInd_K^G (k) \oplus \cInd_K^G (\Sym^{p-1} (k^2)) \simto \cInd_K^G( \cInd_I^K (k) ).
$$

Consider the following two maps $\phi_1, \phi_2 \colon \cInd_I^G (k) \to \cInd_K^G (k)$.
The first is given simply by $\phi_1([g, a]) = [g, a]$.
The second map $\phi_2$ is defined by $\phi_2([g, a]) = [g \alpha, a]$.
It will be useful to view this map as the composition of the map $[g, a] \mapsto [g, a] \colon \cInd_I^G (k) \to \cInd_{\alpha K \alpha^{-1}}^G (k)$ and the intertwining isomorphism
\begin{align}\label{eqn: intertwining isomorphism}
\begin{split}
    \cInd_{\alpha K \alpha^{-1} }^G (k) & \simto \cInd_{K}^G (k) \\
    [g, a] & \mapsto [g \alpha, a].
\end{split}
\end{align}
It is tedious, but straightforward, to check that the resulting maps
$$
    \cInd_K^G( \cInd_I^K (k)) \simto \cInd_I^G ( k ) \rightrightarrows \cInd_K^G (k)
$$
are $k[T, S]$-equivariant.
One can also check that the composition
$$
    \cInd_K^G(k) \oplus \cInd_K^G (\Sym^{p-1}(k^2)^\vee) \to \cInd_K^G( \cInd_I^K (k)) \xrightarrow{\phi_1 \oplus \phi_2} \cInd_K^G(k) \oplus \cInd_K^G(k)
$$
is of the form $\begin{pmatrix}
    1 & 0 \\
    T & \phi
\end{pmatrix},$ where $\phi$ is also $k[T,S]$-equivariant.

\begin{lemma}\label{lemma: map that factors through steinberg}
    The reduction mod $(T-1, S-1)$ of the homomorphism $\phi$ is a non-zero map $\pi(p-1, 1, 1) \to \pi(0, 1, 1)$.
\end{lemma}
\begin{proof}
    The following argument is a representation theoretic analogue of the proof of \cite[Lemma 2]{edixhoven_khare_hasse} (in fact, this lemma can be deduced literally from \loccit, for example as a consequence of \Cref{proposition: hasse invariant map is surjective} below).
    Write ${}^\circ G = \{ g \in G: v_p(\det(g)) = 0 \}$.
    Then, by \cite[II.1.4 Theorem 3]{serre_trees}, ${}^\circ G$ is the amalgamated product of $K$ and $\alpha K \alpha^{-1}$ along $I$.
    Thus, there is a Mayer-Vietoris exact sequence in the group homology of $k[G]$,
    $$
        0 \to \cInd_I^G( k ) \to \cInd_K^G (k) \oplus \cInd_{\alpha K \alpha^{-1}}^G (k) \to \cInd_{{}^\circ G}^G (k) \to 0.
    $$
    Composing with the intertwining isomorphism \Cref{eqn: intertwining isomorphism}, we obtain an exact sequence
    \begin{align}\label{eqn: exact sequence of cInds}
        0 \to \cInd_K^G( \cInd_I^K (k)) \stackrel{\phi_1 \oplus \phi_2}\to \cInd_K^G (k) \oplus \cInd_K^G (k) \to \cInd_{{}^\circ G}^G (k) \to 0
    \end{align}
    where the last map is given by $([g_1, a_1], [g_2, a_2]) \mapsto [g_1, a_1] - [g_2 \alpha^{-1}, a_2]$.
    This exact sequence is $k[T, S]$-equivariant if we endow $\cInd_{{}^\circ G}^G (k)$ with the action of $T$ (resp. $S$) given by acting by $\alpha$ (resp. $\beta$).
    Taking the quotient of the exact sequence above by the ideal $(T-1, \diamondp - 1)$, we obtain an exact sequence
    $$
        \pi(0,1,1) \oplus \pi(p-1, 1, 1) \to \pi(0,1,1) \oplus \pi(0,1,1) \to k \to 0,
    $$
    where the first map is given by $\begin{pmatrix} 1 & 0 \\ 1 & \bar\phi \end{pmatrix}$, where $\bar \phi$ is the map induced by $\phi$.
    Looking at the Jordan-H\"older constituents of the terms in the exact sequence, it is clear that $\bar \phi$ cannot be zero.
\end{proof}

Let us also remark that if $R$ is a $k$-algebra, $\tau \in R, \sigma \in R^\times$ and $s \in \Z$, then tensoring $\phi$ with $\omega^{s} \otimes_k R$ and quotienting by $(T - \tau, S - \sigma)$ we get a map
$\tilde \pi(p-1, \tau, \sigma, s)_R \to \tilde \pi(0, \tau, \sigma, s)_R$.

\subsection{The map \texorpdfstring{$\pi(0,1,1) \to \pi(p-1,1,1)$}{pi(0,1,1) -> pi(p-1,1,1)}}
\label{subsection: map pi 2}

There is also a unique-up-to-scalar non-zero map $\pi(0,1,\chi) \to \pi(p-1,1,\chi)$, which factors through $\chi \circ \det$.
The goal of this section is to show that this map comes from specialising a map $\tilde \pi(0, \tau, \sigma, s)_R \to \tilde \pi(p-1, \tau, \sigma, s)_R$ as in the setting of the end of the previous section.
As in the previous section, this is essentially equivalent to the existence of a lift of the map $\pi(0,1,1) \to \pi(p-1,1,1)$ to a $k[T,S]$-equivariant map
$$
    \cInd_K^G (k) \to \cInd_K^G( \Sym^{p-1} (k^2)^\vee).
$$

We will construct such a map by dualising the procedure of the previous section.
Consider the composition
\begin{align}\label{eqn: map from cIndK2 to cIndI}
    \cInd_K^G(k) \oplus \cInd_K^G(k)
    \xrightarrow{\id \oplus \Cref{eqn: intertwining isomorphism}^{-1}} \cInd_K^G(k) \oplus \cInd_{\alpha K \alpha^{-1}}^G(k) \to \cInd_I^G(k),
\end{align}
where the last map is the sum of inclusions.
It is $k[T,S]$-equivariant and the resulting map
$$
    \cInd_K^G(k) \oplus \cInd_K^G(k)
    \to
    \cInd_K^G(k) \oplus \cInd_K^G(\Sym^{p-1}(k^2)^\vee)
$$
is of the form $\begin{pmatrix}
    1 & S^{-1} T \\
    0 & - \psi
\end{pmatrix}$.
Explicitly, $\psi([g,1]) = \sum_{x \in K / I} [\beta^{-1} g x \alpha, e^*]$, where $e^*$ is the element of $\Sym^{p-1}(k^2)^\vee$ corresponding to the polynomial function $Q(x,y) = x^{p-1}$.

\begin{lemma}\label{lemma: map that factors through trivial}
    The reduction mod $(T - 1, S - 1)$ of the homomorphism $\psi$ is a non-zero map $\pi(0,1,1) \to \pi(p-1,1,1)$.
\end{lemma}
\begin{proof}
    As for \Cref{lemma: map that factors through steinberg}, this follows from another result for group cohomology (namely, \Cref{proposition: steinberg map is zero} below), but we give another proof in a more representation-theoretic spirit.
    Note that $\alpha^{-1} x^{-1} e^*$ is equal to $e^*$ for any $x \in K$ which is not in the same left $I$-coset as $w := \begin{pmatrix}
    0 & 1 \\ 1 & 0
    \end{pmatrix}$, and vanishes if $x \in wI$.
    In particular, $\psi([g,1]) = T [\beta^{-1}, e^*] + [\beta^{-1} w \alpha, e^*]$.
    Hence, it's enough to show that $[1, e^*] + [w \alpha, e^*]$ defines a non-zero element of $\pi(p-1,1,1)$.
    To do this, we will check that its image under the map of \Cref{thm: reminder of mod p representations} \Cref{item: isomorphism to principal series} is non-zero.
    This map is defined in \cite[Section 6.2]{barthel_livne} and sends $[g, Q]$ to $h \mapsto Q(x(1:0))$ where we have written $g^{-1} h = x b$ with $x \in K$ and $b \in B$.
    The image of $[1, e^*] + [w \alpha, e^*]$ maps $w$ to $1$, so in particular it is non-zero (in fact, as we would expect, it is the constant function with value 1).
\end{proof}

\subsection{The resulting maps on arithmetic cohomology}
\label{subsection: maps in arithmetic cohomology}

\begin{proposition}\label{proposition: hasse invariant map is surjective}
    Let $R = k[\tau, \sigma, \sigma^{-1}]$ be as in \Cref{remark: reason for families} and $s \in \Z$.
    Then, the  map
    $\tilde \phi \colon H_*(\Gamma_1(N), \Sym^{p-1} (k^2)^\vee \otimes \omega^{s}) \to H_*(\Gamma_1(N), \omega^{s} \circ \det)$
    induced from the map
    $$
        \tilde \pi(p-1, \tau, \sigma, s)_R \to \tilde \pi(0, \tau, \sigma, s)_R
    $$
    defined at the end of \Cref{subsection: map pi 1}
    is a twist of the dual of the map in \cite[Lemma 2]{edixhoven_khare_hasse}.
    In particular, if $p \geq 5$, this map is surjective in degree 1.
\end{proposition}
\begin{proof}
    By \cite[Corollary 2.11]{buzzard_diamond_jarvis}, we may assume $s = 0$.
    The first sentence follows from the constructions of both maps, and the second follows from \cite[Lemma 2]{edixhoven_khare_hasse}.
\end{proof}

As mentioned above, the proof of \Cref{lemma: map that factors through steinberg} is a representation-theoretic analogue the proof of \cite[Lemma 2]{edixhoven_khare_hasse}.
The latter can then be recovered by taking $p$-arithmetic homology of the exact sequence \Cref{eqn: exact sequence of cInds}.
To see this, we need the following result.

\begin{lemma}
    If $p \geq 5$, the $p$-arithmetic homology $H_1(\Gamma^p_1(N), \cInd_{{}^\circ G}^G (k))$ vanishes.
\end{lemma}
\begin{proof}
    As $G = \Gamma_1^p(N) {}^\circ G$ and $\Gamma_1^p(N) \cap {}^\circ G = \Gamma_1^p(N) \cap \SL_2(\Q)$, the natural restriction map $\cInd_{{}^\circ G}^G (k) \to \cInd_{\Gamma_1^p(N) \cap \SL_2(\Q)}^{\Gamma_1^p(N)} (k)$ is an isomorphism of representations of $\Gamma^p_1(N)$.
    The lemma follows by Shapiro's lemma and \cite[Proof of Lemma 1]{edixhoven_khare_hasse}.
\end{proof}

Thus, when $p \geq 5$, we have a commutative diagram
$$
\begin{tikzcd}[ampersand replacement=\&]
	{H_1(\Gamma_1(N) \cap \Gamma_0(p), k) } \& {H_1(\Gamma_1(N), k) \oplus H_1(\Gamma_1(N), k)} \\
	{H_1(\Gamma_1(N), k) \oplus H_1(\Gamma_1(N), \Sym^{p-1} (k^2)^\vee).}
	\arrow["\sim", from=1-1, to=2-1]
	\arrow[two heads, from=1-1, to=1-2]
	\arrow["{{\begin{pmatrix} \id & 0 \\ T & \tilde \phi \end{pmatrix}}}"'{pos=0.6}, from=2-1, to=1-2]
\end{tikzcd}
$$
In particular, $\tilde \phi$ is surjective, so we have indeed recovered \cite[Lemma 2]{edixhoven_khare_hasse}.
Moreover, we see that the kernel of $\tilde \phi$ is isomorphic to the kernel of the map
\begin{align}\label{eqn: map that I want to know the kernel of}
    H_1(\Gamma_1(N) \cap \Gamma_0(p), k) \onto H_1(\Gamma_1(N), k) \oplus H_1(\Gamma_1(N), k).
\end{align}
This is the homomorphism induced by the maps between open modular curves determined by $\tau \mapsto \tau$ and $\tau \mapsto p \tau$ on the upper-half plane.
A generalisation by Wiles of a lemma of Ribet determines the Galois representations that contribute to this kernel.

\begin{proposition}\label{lemma: ribets lemma}
    Let $\rho$ be a 2-dimensional odd irreducible representation of $\Gal(\bar \Q / \Q)$ over $k$ such that $N(\rho)$ divides $N$ and $\rho|_{\Gcal_p} \simeq \begin{pmatrix} 1 & * \\ 0 & \omega^{-1} \end{pmatrix} \otimes \unr{b}$.
    Then, the localisation at the maximal ideal $\m_\rho$ of $\Tbb(pN)$ corresponding to $\rho$ of the kernel of \Cref{eqn: map that I want to know the kernel of} is non-zero.
\end{proposition}
\begin{proof}
    If the extension $*$ in the statement is tr\`es ramifi\'ee, then it is clear that $\rho$ contributes to the kernel as it contributes to the source but not the target.
    Therefore, we may assume that the extension is peu ramifi\'ee.
    Let $f$ be a normalised newform of weight 2, level $\Gamma_1(N)$ and character $\chi$ whose associated Galois representation over $k$ is isomorphic to $\rho$.
    If $a_p$ is its $p$-th Fourier coefficient, let $\alpha$ be the root of $x^2 - a_p x - \chi(p) p$ that is a $p$-adic unit (there should be no confusion with our previous use of the letter $\alpha$). Then $\alpha^2 \equiv a_p^2 \equiv \chi(p) \mod p$, the second congruence following from \cite[Theorem 2.5]{edixhoven_weight}.

    There is an eigenclass in $H_1(\Gamma_1(N), \bar \Z_p)$ for the Hecke operators away from $N$ with the same system of eigenvalues as $f$ and whose reduction to $H_1(\Gamma_1(N), k)$ is non-zero.
    As $\rho$ is irreducible, the localisation of $H_1(\Gamma_1(N), k)$ at $\m_\rho$ is isomorphic to $(k \otimes_{\F_p} J_1(N)[p])_{\m_\rho}$, where $J_1(N)$ is the Jacobian of the compactified modular curve of level $\Gamma_1(N)$.
    Write $P$ for the element of $(k \otimes_{\F_p} J_1(N)[p])_{\m_\rho}$ corresponding to the reduction of the eigenclass above.
    Similarly, $H_1(\Gamma_1(N) \cap \Gamma_0(N), k)_{\m_\rho} \simeq (k \otimes_{\F_p} J_1(N, p)[p])_{\m_\rho}$, where $J_1(N, p)$ is the Jacobian of the compactified modular curve of level $\Gamma_1(N) \cap \Gamma_0(p)$.
    Consider now the image of $(P,0) = (P, -(a_p - \alpha) P)$ under the map
    $$
        (k \otimes_{\F_p} J_1(N)[p])_{\m_\rho}
        \oplus
        (k \otimes_{\F_p} J_1(N)[p])_{\m_\rho}
        \to
        (k \otimes_{\F_p} J_1(N, p)[p])_{\m_\rho}
    $$
    induced by the morphisms between modular curves determined by $\tau \mapsto \tau$ and $\tau \mapsto p \tau$ on the upper-half plane.
    The image of $(P, 0)$ is then one of the $p$-stabilisations of $P$:
    it is an eigenvector for all Hecke operators $T_\ell$ for $ \ell \nmid pN$ (with eigenvalues determined by $\rho$), as well as the operator $U_p$ (resp. $\langle n \rangle$ for any $n \in \Z$ with $n \equiv p \mod N$ and $n \equiv 1 \mod p$) with eigenvalue $\alpha$ (resp. $\chi(p)$).
    It is also non-zero (for example, by \Cref{proposition: steinberg map is zero} and its proof below, or by the injectivity of \cite[(2.10)]{wiles_fermat}).
    Hence, by \cite[Lemma 2.3]{wiles_fermat}, it defines (under the isomorphisms above between group homology and $p$-torsion points in Jacobians) a non-zero element of the localisation at $\m_\rho$ of the kernel of \Cref{eqn: map that I want to know the kernel of}.
\end{proof}

Finally, we turn to the map from \Cref{subsection: map pi 2}.

\begin{proposition}\label{proposition: steinberg map is zero}
    Let $R = k[\tau, \sigma, \sigma^{-1}]$ be as in \Cref{remark: reason for families} and $s \in \Z$. Let $\rho$ be a 2-dimensional odd irreducible representation.
    Then, the  map
    $$
        H_*(\Gamma_1(N), \omega^{s} \circ \det)_{\rho} \to H_*(\Gamma_1(N), \Sym^{p-1} (k^2)^\vee \otimes_k \omega^{s})_{\rho}
    $$
    induced from the map
    $
        \tilde \pi(0, \tau, \sigma, s)_R \to \tilde \pi(p-1, \tau, \sigma, s)_R
    $
    is injective in degree 1.
    If $N(\rho)$ divides $N$ and $\rho|_{\Gcal_p} \simeq \begin{pmatrix} 1 & * \\ 0 & \omega^{-1} \end{pmatrix} \otimes \omega^s \unr{b}$, then the cokernel is non-zero.
\end{proposition}
\begin{proof}
    The proposition follows from \Cref{proposition: hasse invariant map is surjective} and \Cref{lemma: ribets lemma} by Poincar\'e duality.
    Let us spell out the details.
    Again, we may assume that $s = 0$.
    Given a $\Tbb(pN)$-module $M$, let us write $M^*$ for the base change of $M$ along the ring isomorphism $\Tbb(pN) \to \Tbb(pN)$ mapping $T_\ell \mapsto S_\ell^{-1} T_\ell$ and $S_\ell \mapsto S_\ell^{-1}$.
    Thus, $\rho$ contributes to $M$ if and only if $\rho^\vee \otimes \varepsilon^{-1}$ contributes to $M^*$.
    Let us also write $M^*_\rho := (M^*)_\rho$.
    As $\rho$ is irreducible, there are Poincar\'e duality isomorphisms
    \begin{align*}
        H_1(\Gamma_1(N), k)_{\rho} & \simto H^1(\Gamma_1(N), k)^*_{\rho} \\
        H_1(\Gamma_1(N), \Sym^{p-1} (k^2)^\vee)_{\rho} & \simto H^1(\Gamma_1(N),  \Sym^{p-1} (k^2)^\vee)^*_{\rho} \\
        H_1(\Gamma_1(N) \cap \Gamma_0(p), k)_{\rho} & \simto H^1(\Gamma_1(N) \cap \Gamma_0(p), k)^*_{\rho}.
    \end{align*}
    On the one hand, the $k$-linear dual of \Cref{eqn: map that I want to know the kernel of} is the direct sum of the pullbacks in cohomology of the maps from the modular curve of level $\Gamma_1(N) \cap \Gamma_0(p)$ to the modular curve of level $\Gamma_1(N)$ determined by $\tau \mapsto \tau$ and $\tau \mapsto p \tau$ on the upper-half plane.
    On the other hand, the map \Cref{eqn: map from cIndK2 to cIndI} induces on $p$-arithmetic homology a map
    $$
        H_1(\Gamma_1(N), k) \oplus H_1(\Gamma_1(N), k) \to H_1(\Gamma_1(N) \cap \Gamma_0(p), k),
    $$
    that is (by construction) the direct sum of the transfer homomorphisms in the homology of the modular curves above induced by the same pair of maps.
    Now, under Poincar\'e duality, pullbacks in cohomology correspond to transfer homomorphisms in homology, and thus the two maps correspond (after localising at $\rho$) under the above Poincar\'e duality isomorphism.
    
    Fix an isomorphism $\Sym^{p-1}(k^2)^\vee \simto \Sym^{p-1}(k^2)$.
    It determines also an isomorphism between $\Map(\P^1(\F_p), k)$ and its dual, and by Shapiro's lemma an isomorphism
    $$
        H^1(\Gamma_1(N) \cap \Gamma_0(p), k)
        \simto
        H^1(\Gamma_1(N), k)
        \oplus
        H^1(\Gamma_1(N), \Sym^{p-1} (k^2)^\vee)
    $$
    that is compatible under Poincar\'e duality with the similar isomorphism for homology.
    In conclusion, we have a commutative square
    $$
    \begin{tikzcd}
        H_1(\Gamma_1(N), k)_{\rho} \oplus H_1(\Gamma_1(N), k)_{\rho}
        \ar[r] \ar[d, "\sim"]
        &
        H_1(\Gamma_1(N), k)_{\rho} \oplus H_1(\Gamma_1(N),  \Sym^{p-1} (k^2)^\vee)_{\rho}  \ar[d, "\sim"]
        \\
        H^1(\Gamma_1(N), k)^*_{\rho} \oplus H^1(\Gamma_1(N), k)^*_{\rho}
        \ar[r]
        &
        H^1(\Gamma_1(N), k)^*_{\rho} \oplus H^1(\Gamma_1(N),  \Sym^{p-1} (k^2)^\vee)^*_{\rho}
    \end{tikzcd}
    $$
    where the top horizontal map is obtained from taking the $p$-arithmetic homology of \Cref{eqn: map from cIndK2 to cIndI} and the bottom horizontal map is (the Poincar\'e dual of) the $k$-linear dual of \Cref{eqn: map that I want to know the kernel of}.
    Thus, the map $H_1(\Gamma_1(N), k)_{\rho} \to H_1(\Gamma_1(N), \Sym^{p-1} (k^2)^\vee)_{\rho}$ from the statement of the lemma corresponds under these isomorphisms to the dual of the map from \Cref{proposition: hasse invariant map is surjective}.
    Thus, dualising \Cref{proposition: hasse invariant map is surjective} and \Cref{lemma: ribets lemma} gives the result.
\end{proof}

\section{Proof of \texorpdfstring{\Cref{thm: 1}}{Theorem 1.1}}

In \Cref{subsection: proof generic} we have proven the generic case of \Cref{thm: 1}, in this section we are going to deal with the non-generic cases of twists of the trivial and Steinberg representations.
We assume throughout that $p \geq 5$.

\subsection{The Steinberg case}
\label{subsection: steinberg case}

The case of twists of the Steinberg representation will follow from \Cref{proposition: representations in weights 0 and p-1} and \Cref{proposition: steinberg map is zero} and some formal algebraic manipulations.
Consider a two-term complex $C_1 \to C_0$ of representations of $G$ such that $H_0(C_\bullet) \simeq H_1(C_\bullet)$.
Taking the $p$-arithmetic hyperhomology of this complex induces two spectral sequences converging to the same abutment, one $E$ with $E^2$ page given by
$$
    E^2_{i,j} = H_i(\Gamma^p_1(N), H_j(C_\bullet))
$$
and the other ${}' E$ with ${}' E^1$ page given by
$$
    {}' E^1_{i,j} = H_j(\Gamma^p_1(N), C_i).
$$
This spectral sequence degenerates at the ${}'E^2$ page, so the systems of Hecke eigenvalues appearing in this page are the same as those in the abutment.
A spectral sequence argument shows that these systems of Hecke eigenvalues are the same as those appearing in $E^2$.
Indeed, if $i_0$ is the smallest degree for which a fixed system of eigenvalues appears in $H_{i_0}(\Gamma^p_1(N), H_0(C_\bullet))$, then the localisation at this system of eigenvalues of the $E_{i_0, 0}$ term is stable in the localised spectral sequence, so the localisation of the abutment in degree $i_0$ will be non-zero.
Moreover, if the homologies $H_j(\Gamma^p_1(N), C_i)$ are finite-dimensional, then so is the abutment, and the same type of spectral sequence argument shows that so are the terms in $E^2_{i,j}$.

Let us now specialise to our case of interest. The complex we will be considering is given by the map
$$
    C_1 := \pi(0, 1, \chi) \to \pi(p-1, 1, \chi) := C_0
$$
from \Cref{lemma: map that factors through trivial}, so that $H_0(C_\bullet) \simeq H_1(C_\bullet) \simeq \St \otimes \chi$.
Thus, the previous paragraph and \Cref{proposition: representations in weights 0 and p-1} show that \Cref{thm: 1} \Cref{item: thm 1 i} and \Cref{item: thm 1 ii} are satisfied for twists of the Steinberg representation.
Let us analyse the systems of eigenvalues appearing in homology.
\Cref{proposition: representations in weights 0 and p-1} shows that the only odd irreducible Galois representations $\rho$ which can contribute to the $'E^1$ page above, and hence to $H_*(\Gamma^p_1(N), \St \otimes \chi)$, are those such that $N(\rho)$ divides $N$ and $\rho|_{\Gcal_p}$ is an extension of $\chi \omega^{-1}$ by $\chi$.
Fix such a $\rho$ and write $\chi = \omega^a \unr{b}$.
By the argument in \cite[Proposition 5.8]{S_arithmetic_cohomology}, and using that derived tensor products commute with mapping cones, the $p$-arithmetic hyperhomology of $C_\bullet$ is isomorphic (in the derived category of $\Tbb(pN) \otimes_\Z k[T, S]$-modules) to the derived tensor product over $k[T,S]$ of $k[T,S]/(T-b, S-b^2)$ and
\begin{align*}
    & \left[
        C_\bullet(\Gamma_1^p(N), \cInd_K^G(\omega^a))
        \to
        C_\bullet(\Gamma_1^p(N), \cInd_K^G(\Sym^{p-1}(k^2)^\vee \otimes \omega^a))
    \right] \\
    & \simeq
    \left[
        C_\bullet(\Gamma_1(N), \omega^a \circ \det)
        \to
        C_\bullet(\Gamma_1(N), \Sym^{p-1}(k^2)^\vee \otimes \omega^a)
    \right],
\end{align*}
where $\left[ \blank \right]$ denotes mapping cones and the isomorphism follows from Shapiro's lemma \cite[Proposition 5.3]{S_arithmetic_cohomology}.
Using that $\rho$ contributes to arithmetic homology only in degree 1, we see that the localisation of the above hyperhomology at $\rho$ is
\begin{align*}
    &
    \frac{k[T,S]}{(T-b, S-b^2)} \otimesL_{k[T,S]}  \left[
        H_1(\Gamma_1(N), \omega^a \circ \det)_\rho[1]
        \to
        H_1(\Gamma_1(N), \Sym^{p-1}(k^2)^\vee \otimes \omega^a)_\rho[1]
    \right] \\
    & \simeq
    \frac{k[T,S]}{(T-b, S-b^2)} \otimesL_{k[T,S]} U_{a, \rho}[1]
\end{align*}
where the map in the first line is that of \Cref{proposition: steinberg map is zero} and
$$
    U_a = \coker \left( H_1( \Gamma_1(N), \omega^a \circ \det ) \to H_1( \Gamma_1(N), \Sym^{p-1} (k^2)^\vee \otimes \omega^a ) \right).
$$
In conclusion, the abutment of the localisation spectral sequences $E$ and $'E$ is given in degree $i$ by
$$
     \Tor^{k[T,S]}_{i-1} \left( \frac{k[T,S]}{(T-b, S-b^2)}, U_{a, \rho} \right)
$$
In particular, analysing the $E^2$ page shows that
\begin{align*}
    H_0(\Gamma^p_1(N), \St \otimes \chi)_\rho & = H_3(\Gamma^p_1(N), \St \otimes \chi)_\rho = 0, \\
    H_1(\Gamma^p_1(N), \St \otimes \chi)_\rho & \simeq \frac{k[T, S]}{(T-b, S - b^2)} \otimes_{k[T,S]} U_{a,\rho}, \\
    H_2(\Gamma^p_1(N), \St \otimes \chi)_\rho & \simeq \Hom_{k[T,S]} \left( \frac{k[T, S]}{(T-b, S - b^2)}, U_{a, \rho} \right).
\end{align*}
Moreover, \Cref{proposition: steinberg map is zero} shows that the homology in degrees 1 and 2 is always non-zero, since $\rho$ contributes to $U_a$ and can only contribute to its $(T=b, S = b^2)$-eigenspace by \cite[Theorem 2.5]{edixhoven_weight}.
Together with \Cref{proposition: socle llc generic reducible}, this completes the proof of \Cref{thm: 1} \Cref{item: thm 1 iii} in this case.

\subsection{The trivial case}

The case of twists of the trivial representation is completely analogous to the case of twists of the Steinberg representation, this time applying the above analysis to
$$
    C_1 := \pi(p-1, 1, \chi) \to \pi(0, 1, \chi) := C_0,
$$
where the map is that of \Cref{lemma: map that factors through steinberg}.
It satisfies $H_0(C_\bullet) \simeq H_1(C_\bullet) \simeq \chi \circ \det$, so \Cref{thm: 1} \Cref{item: thm 1 i} and \Cref{item: thm 1 ii} hold for these representations.
The corresponding spectral sequences and \Cref{proposition: hasse invariant map is surjective} show that if $\chi = \omega^a \unr{b}$ and $\rho$ is an odd irreducible representation of $\Gal(\bar \Q / \Q)$,
\begin{align*}
    H_0(\Gamma^p_1(N), \chi \circ \det)_\rho & = H_1(\Gamma^p_1(N), \chi \circ \det)_\rho = 0, \\
    H_2(\Gamma^p_1(N), \chi \circ \det)_\rho & \simeq \frac{k[T, S]}{(T-b, S - b^2)} \otimes_{k[T,S]} V_{a,\rho}, \\
    H_3(\Gamma^p_1(N), \chi \circ \det)_\rho & \simeq \Hom_{k[T,S]} \left( \frac{k[T, S]}{(T-b, S - b^2)}, V_{a, \rho} \right).
\end{align*}
where
$$
V_{a} := \ker \left( H_1( \Gamma_1(N), \Sym^{p-1} (k^2)^\vee \otimes \omega^a ) \to H_1( \Gamma_1(N), \omega^a \circ \det ) \right).
$$
By \Cref{lemma: ribets lemma} (and \cite[Corollary 2.11]{buzzard_diamond_jarvis}), the odd irreducible Galois representations $\rho$ contributing to the $(T=b, S=b^2)$-eigenspace of $V_a$ are those such that $N(\rho)$ divides $N$ and $\rho|_{\Gcal_p}$ is an extension of $\chi \omega^{-1}$ by $\chi$.
These representations therefore appear in the $p$-arithmetic homology of $\chi \circ \det$ exactly in degrees 2 and 3.
There are no Galois representations satisfying condition \Cref{item: thm 1 iii a} of \Cref{thm: 1} \Cref{item: thm 1 iii}, as finite-dimensional representations never appear in the socle of representations in the image of the mod $p$ local Langlands correspondence for $\GL_2(\Q_p)$.
Therefore, \Cref{thm: 1} \Cref{item: thm 1 iii} holds, which completes the proof of \Cref{thm: 1}.

\begin{remark}
    In the proof of \Cref{proposition: steinberg map is zero} we showed using Poincar\'e duality for arithmetic cohomology that $U_a = (V_{-a}^\vee)^*$ (with the notation introduced in that proof).
    The Poincar\'e duality isomorphisms in that proof intertwine $T$ with $S^{-1}T$ and $S$ with $S^{-1}$.
    This, together with the above computations, implies that for $\rho$ as above there are ``Poincar\'e duality" isomorphisms
    \begin{align*}
        H^i(\Gamma_1^p(N), \chi \circ \det)_\rho & \simto H_{4-i} (\Gamma_1^p(N), \St \otimes \chi)^*_\rho.
    \end{align*}
\end{remark}

\begin{remark}
    Assume that $\chi$ is the trivial character for simplicity. Then, the fact that Galois representations as above contribute to $H_*(\Gamma_1^p(N), k)$ and $H^*(\Gamma_1^p(N), k)$ but not in degree 1 (unlike in the other cases) is related to the fact that, if $\pi$ is the smooth representation of $\GL_2(\Q_p)$ corresponding to $\rho|_{\Gcal_p}$ under the mod $p$ Langlands correspondence, then $\Hom_{k[G]}(k, \pi) = 0$, but $\Ext^1_{k[G]}(k, \pi) \neq 0$ (at least when $\rho|_{\Gcal_p}$ is non-split).
    This can be made precise by relating $p$-arithmetic cohomology to completed cohomology and taking into account Emerton's local-global compatibility results \cite{emerton_local_global}, but we do not pursue this here.
\end{remark}

\bibliographystyle{alpha}
\bibliography{references}

\end{document}